\documentclass[11pt,twoside]{amsart}
\usepackage{amsmath}
\usepackage{amssymb}
\usepackage{amsthm}
\usepackage{latexsym}
\usepackage{amscd}
\usepackage{xypic}
\xyoption{curve}
\usepackage{ifthen} 
\usepackage{hyperref} 
\usepackage{graphicx}
\usepackage{enumerate}

 \def\cocoa{{\hbox{\rm C\kern-.13em o\kern-.07em C\kern-.13em o\kern-.15em A}}}

\usepackage{xcolor}

\newtheorem{theorem}{Theorem}[section]
\newtheorem{lemma}[theorem]{Lemma}
\newtheorem{proposition}[theorem]{Proposition}
\newtheorem{corollary}[theorem]{Corollary}

\theoremstyle{definition}
\newtheorem{remark}[theorem]{Remark}
\newtheorem{definition}[theorem]{Definition}

\newcommand {\pf}{\mathrm{pf}}

\newcommand {\gr}{\mathrm{gr}}

\newcommand {\Hom}{\mathcal{H}\kern -0.25ex{\mathit om}}
\newcommand {\Spl}{\mathcal{S}\kern -0.25ex{\mathit pl}}
\newcommand {\Ext}{\mathcal{E}\kern -0.25ex{\mathit xt}}

\newcommand {\rk}{\mathrm{rk}}

\newcommand {\ext}{\mathrm{Ext}}
\newcommand {\Hilb}{\mathcal{H}\kern -0.25ex{\mathit ilb\/}}

\newcommand {\bZ}{\mathbb{Z}}

\newcommand {\bP}{\mathbb{P}}

\newcommand{\cE}{{\mathcal E}}
\newcommand{\cF}{{\mathcal F}}
\newcommand{\cM}{{\mathcal M}}

\newcommand{\cO}{{\mathcal O}}
\newcommand{\cG}{{\mathcal G}}
\newcommand{\cI}{{\mathcal I}}

\newcommand{\Pic}{\operatorname{Pic}}
\newcommand{\NS}{\operatorname{NS}}
\newcommand{\Num}{\operatorname{Num}}

\def\p#1{{\bP^{#1}}}

\def\mapright#1{\mathbin{\smash{\mathop{\longrightarrow}
\limits^{#1}}}}

\title[Examples of rank two aCM bundles on smooth quartic surfaces in $\p3$]{Examples of rank two aCM bundles\\ on smooth quartic surfaces in $\p3$}



\author[Gianfranco Casnati, Roberto Notari]{Gianfranco Casnati, Roberto Notari}
\thanks{The authors are members of the GNSAGA group of INdAM and are supported by the framework of PRIN 2010/11 \lq Geometria delle variet{\accent"12 a} algebriche\rq, cofinanced by MIUR}

\begin{document}

\dedicatory{This paper is dedicated to Ph. Ellia on the occasion of his 60$\,{}^{ th}$ birthday.}

\begin{abstract}
Let $F\subseteq\p 3$ be a smooth quartic surface  and let $\cO_F(h):=\cO_{\p 3}(1)\otimes\cO_F$. In the present paper we classify locally free sheaves $\cE$ of rank $2$ on $F$
such that $c_1(\cE)=\cO_F(2h)$, $c_2(\cE)=8$ and $h^1\big(F,\cE(th)\big)=0$ for $t\in\bZ$. We also deal with their stability.

\end{abstract}

\maketitle

\section{Introduction and Notation}
We work over an algebraically closed field $k$ of characteristic $0$ and $\p 3$ will denote the projective space over $k$. 

Let $F\subseteq\p 3$ be a smooth surface and let $\cO_F(h):=\cO_{\p 3}(1)\otimes\cO_F$. A vector bundle $\cE$ on $F$ is called {\sl arithmetically Cohen--Macaulay} ({\sl aCM} for short) if $h^1\big(F,\cE(th)\big)=0$ for $t\in\bZ$. 

The property of being aCM is invariant up to shifting degrees. For this reason we restrict our attention to {\sl initialized} bundles, i.e. bundles $\cE$ such that $h^0\big(F,\cE(-h)\big)=0$ and $h^0\big(F,\cE\big)\ne0$. Moreover, we are also interested in {\sl indecomposable} bundles, i.e. bundles which do not split in a direct sum of bundles of lower rank.

Clearly $\cO_F$ is always initialized, indecomposable and aCM. Moreover, when $F=\p2$ the line bundle $\cO_F$ is the unique bundle with these properties: this is the so called Horrocks theorem (see \cite{O--S--S} and the references therein). 

As a by--product of a more general result of D. Eisenbud and D. Herzog (see \cite{E--He}) there is only another surface in $\p3$ endowed with at most a finite number of aCM bundles besides the plane, namely the smooth quadric surface $Q$. In this case there are only three indecomposable, initialized aCM bundles on $Q$ and they are all line bundles ($\cO_F$ and the so--called spinor bundles: see \cite{Ot} for further details). 
On the opposite side M. Casanellas and R. Hartshorne proved in \cite{C--H1} and \cite{C--H2} that a smooth cubic surface is endowed with families of arbitrary dimension of non--isomorphic, indecomposable initialized aCM bundles.
It is thus interesting to better understand indecomposable initialized aCM bundles. 

The study of aCM line bundles is trivially strictly related with the study of $\Pic(F)$. As pointed out above the results for the plane and the smooth quadric are classical. D. Faenzi completely described the case of a smooth cubic surface in the very first part of \cite{Fa}. K. Watanabe made a similar study for the smooth quartic surface (see \cite{Wa}). More in general, the existence of an aCM line bundle which has the maximal number of generators on a surface $F$ is related to the possibility of writing the equation of $F$ as the determinant of a square matrix with linear entries (see e.g. \cite{Ct} or \cite{Bea} and the references therein).

The next step is to deal with indecomposable, initialized aCM bundles of rank $2$. We already know that this study is meaningful only if $\deg(F)\ge3$: the main result of the aforementioned paper \cite{Fa} is the full classification of these bundles when $\deg(F)=3$. 

When $\Pic(F)$ is generated by $\cO_F(h)$ and $\cE$ is an  indecomposable, initialized aCM bundle of rank $2$ on $F$ with $c_1(\cE)\cong \cO_F(ch)$, then  the restriction $3-\deg(F)\le c\le \deg(F)-1$ holds: the bound was essentially proved by C. Madonna
 (see \cite{Ma1} and \cite{Cs1}). The same result holds also without restriction on $\Pic(F)$, if we consider bundles $\cE$ with $c_1(\cE)\cong \cO_F(ch)$ and such that the zero--loci of their general sections have pure codimension $2$ (see \cite{Cs1}). 
Notice that many other indecomposable, initialized aCM bundles of rank $2$ not satisfying the restriction $c_1(\cE)\cong \cO_F(ch)$ can actually exist on $F$ when $\cO_F(h)$ does not generate $\Pic(F)$ as the results proved in \cite{Fa} show.

Nevertheless the study of indecomposable, initialized aCM bundles $\cE$ of rank $2$ on a surface $F\subseteq\p3$ satisfying the above property on their first Chern class is of intrinsic interest. 

Bundles with $c$ either $3-\deg(F)$, or $4-\deg(F)$ are easy to construct via the so called Serre correspondence. The problem of the existence of bundles with $c=\deg(F)-1$ is a challenging and intriguing problem. It is surprisingly related to the possibility of writing the equation of $F$ as the pfaffian of a skew--symmetric matrix with linear entries (see again \cite{Ct} or \cite{Bea}).

As explained above, the first non--completely known case is $\deg(F)=4$. In this case we know that if  $c_1(\cE)\cong \cO_F(ch)$ and the zero--locus of the general section of $\cE$ has pure codimension $2$, then $-1\le c\le 3$. 

On the one hand, E. Coskun, R.S. Kulkarni, Y. Mustopa completed recently the analysis of the case $c=3$ in \cite{C--K--M1}. They succeed to prove therein that each smooth quartic $F$ supports a family of dimension $14$ of simple (and often stable) indecomposable, initialized, aCM bundles $\cE$ of rank $2$ with $c_1(\cE)=\cO_F(3h)$ and $c_2(\cE)=14$. As a by--product they are able to prove that the polynomial $f$ defining a smooth quartic surface in $\p3$ is the pfaffian of a skew--symmetric matrix of linear forms.

On the other hand, the cases $c=-1,0$ are well--known and the case $c=1$ is not difficult to describe. In the present paper we focus our attention on the case $c=2$. In \cite{Cs2} we will deal with aCM bundles on the general determinantal quartic surface in $\p3$.

In Theorems \ref{t2Exotic}, \ref{t2gg} and \ref{t2ngg}, we show that each smooth quartic $F$ supports indecomposable, initialized, aCM bundles $\cE$ of rank $2$ with $c_1(\cE)=\cO_F(2h)$ and $c_2(\cE)=8$. Such bundles have a minimal free resolution over $\p3$ of the form
\begin{equation}
\label{resGeneral}
0\longrightarrow \cO_F(-2)^{\oplus4}\oplus\cO_F(-1)^{\oplus n-4}\mapright{\varphi}\cO_F^{\oplus4}\oplus\cO_F(-1)^{\oplus n-4}\longrightarrow\cE\longrightarrow0,
\end{equation}
where $n=4,6,8$ and $\varphi$ has skew--symmetric matrix $\Phi$ whose pfaffian $\pf(\Phi)$ defines an equation of $F$, i.e. $F=\{\ \pf(\Phi)=0\ \}$. 

While the cases $n=4,6$ occur on each smooth quartic surface, only determinantal quartics are endowed with bundles with $n=8$. As a by--product of the existence of bundles with $n=4$, one easily deduces that the polynomial $f$ defining a smooth quartic surface in $\p3$ is the pfaffian of a skew--symmetric matrix of quadratic forms, i.e. $f=q_1q_2+q_3q_4+q_5q_6
$ for suitable $q_1,\dots,q_6\in H^0\big(\p3,\cO_{\p3}(2)\big)$.

It is interesting to notice that while in the cases $n=4,6$ the zero--locus of a general section of $\cE$ has pure codimension $2$, in the case $n=8$ each non--zero section vanishes on a divisor. 

We conclude the paper by dealing with the stability properties of these bundles, showing, in particular, that the moduli space of initialized, indecomposable, aCM, stable rank $2$ bundles $\cE$ with $c_1(\cE)=\cO_F(2h)$ and $c_2(\cE)=8$ on $F$ is always non--empty (see Theorem \ref{tStable}).

The authors would like to thank E. Carlini and F. Galluzzi for several helpful discussions. We also thank the referee, whose suggestions allow us to considerably improve the exposition.

For all the notations and unproven results we refer to \cite{Ha2}.

\section{Preliminary results}
\label{sPreliminary}
From now on, in what follows, $F$ will denote an integral smooth quartic surface in $\p3$. Moreover, we will assume that $F=\{\ f=0\ \}$ where $f\in S:=k[x_0,x_1,x_2,x_3]$ is a suitable quartic form.

We are interested in dealing with initialized aCM bundles $\cE$ of rank $2$ with $c_1(\cE)=\cO_F(2h)$ on $F$. The first part of Theorem B of \cite{Bea} implies the existence of a minimal free resolution of $\cE$ of the form
$$
0\longrightarrow\bigoplus_{j=1}^n\cO_{\p3}(d_j-2)\mapright\varphi\bigoplus_{i=1}^n\cO_{\p3}(-d_i)\longrightarrow\cE\longrightarrow0
$$
where we assume $0\le d_i\le d_{i+1}$ for $i=1,\dots,n-1$, and $\varphi$ is defined by a skew--symmetric matrix $\Phi$ such that $\pf(\Phi)=f$. Notice that the $(i,j)$--entry of $\Phi$ is a form of degree $2-d_i-d_j$. In particular $\pf(\Phi)$ is a form of degree $n-\sum_{i=1}^nd_i$, whence
\begin{equation}
\label{number1}
n-4=\sum_{i=1}^nd_i.
\end{equation}

Since $h^0\big(F,\cE\big)\ne0$, it follows that $d_1=0$. Moreover, the minimality of the above resolution implies that all the non--zero entries have degree at least $1$. Thus, in order to avoid that the $n^{\mathrm {th}}$ row in $\Phi$ vanishes, we must have $2-d_n\ge1$ hence $d_n\le1$. We conclude that the $d_j$'s are either $0$ or $1$. Equality \eqref{number1} implies that $d_1=d_2=d_3=d_4=0$ and $d_j=1$ when $j\ge5$. In particular
\begin{equation}
\label{matrixPhi}
\Phi=\left(\begin{array}{cc}
B&A\\
-{}^tA&0
\end{array}
\right)
\end{equation}
where $B$ is a $4\times4$ skew--symmetric matrix with quadratic entries, $A$ is a $4\times (n-4)$ matrix with linear entries and $0$ is the $(n-4)\times(n-4)$ zero--matrix.

We deduce that $\rk(\Phi)\le\rk(B\ A)+\rk(A)\le 8$ at each point of $\p3$. Since $\det(\Phi)=f^2\ne0$, it follows that $n\le8$ and it is even. Thus the minimal free resolution of $\cE$ is Resolution \eqref{resGeneral} and $F=\{\ \pf(\Phi)=0\ \}$. The above discussion proves the first part of the following statement.

\begin{proposition}
\label{pResolution}
Let $F\subseteq\p3$ be a smooth quartic surface.

A sheaf $\cE$ is an initialized, aCM rank $2$ bundle on $F$ with $c_1(\cE)=\cO_F(2h)$ and $c_2(\cE)=8$ if and only if its minimal free resolution on $\p3$ is Resolution \eqref{resGeneral} with $n=4,6,8$ and $\varphi$ is skew--symmetric.  
\end{proposition}
\begin{proof}
We have to prove that if the minimal free resolution of $\cE$ on $F$ is Resolution \eqref{resGeneral} with $n=4,6,8$, then $\cE$ is an initialized, aCM bundle of rank $2$ on $F$ with $c_1(\cE)=\cO_F(2h)$ and $c_2(\cE)=8$.

Clearly $h^0\big(F,\cE\big)=4$ and $h^0\big(F,\cE(-h)\big)=0$, thus $\cE$ is initialized. The second part of Theorem B of \cite{Bea} implies that $\cE$ is an aCM torsion--free sheaf with $c_1(\cE)\cong\cO_F(2h)$, hence it is a bundle because $F$ is smooth. Thus $h^1\big(F,\cE(-h)\big)=0$ and $h^2\big(F,\cE(-h)\big)=h^0\big(F,\cE^\vee(h)\big)=h^0\big(F,\cE(-h)\big)=0$. Since $\det(\Phi)$ is the square of the equation of $F$, it follows that $\rk(\cE)=2$. 

For each locally free sheaf $\cF$ of rank $r$ on $F$, Riemann--Roch theorem gives
\begin{equation}
\label{RRgeneral}
\chi(\cF)=2r+\frac{c_1(\cF)^2}2-c_2(\cF).
\end{equation}
Thanks to the above computations we know that $\chi(\cE(-h))=0$, hence  the above formula finally yields $c_2(\cE)=8$.
\end{proof}

\section{Line bundles on smooth quartic surfaces}
\label{sQuartic}
In this section we will recall some facts about line bundles on a smooth quartic surface $F\subseteq\p3$, most of them collected from several papers (see \cite{C--K--M2}, \cite{SD}, \cite{Wa}: we also mention the book \cite{Hu}). 

By adjunction in $\p3$ we have $\omega_F\cong\cO_F$, i.e. each such an $F$ is a $K3$ surface, thus we can make use of all the results proved in \cite{SD}. The first important fact is that Serre duality on $F$ becomes $h^i\big(F,\cF\big)=h^{2-i}\big(F,\cF^\vee\big)$, $i=0,1,2$, for each locally free sheaf $\cF$ on $F$. 
If $D$ is an effective divisor on $F$, then $h^2\big(F,\cO_F(D)\big)=h^0\big(F,\cO_F(-D)\big)=0$. Moreover
$$
h^1\big(F,\cO_F(D)\big)=h^1\big(F,\cO_F(-D)\big)=h^0\big(D,\cO_D\big)-1
$$
(see \cite{SD}, Lemma 2.2). It follows that for each irreducible effective divisor $D$ on $F$, the dimension of $\vert D\vert$ and the arithmetic genus of $D$ satisfy
\begin{equation}
\label{h^1}
h^0\big(F,\cO_F(D)\big)=2+\frac {D^2}2,\qquad p_a(D)=1+\frac {D^2}2,
\end{equation}
(see \cite{SD}, Paragraph 2.4).

We summarize the most helpful results in the following statement.

\begin{proposition}
\label{SD}
Let $F\subseteq\p3$ be a smooth quartic surface.

For each effective divisor $D$ on $F$ such that $\vert D\vert$ has no fixed components the following assertions hold. 
\begin{enumerate}
\item $D^2\ge0$ and $\cO_F(D)$ is globally generated.
\item If $D^2>0$, then the general element of $\vert D\vert$ is irreducible and smooth: in this case  $h^1\big(F,\cO_F(D)\big)=0$.
\item If $D^2=0$, then there is an irreducible curve $\overline{D}$ with $p_a(\overline{D})=1$ such that $\cO_F(D)\cong\cO_F(e\overline{D})$ where $e-1:=h^1\big(F,\cO_F(D)\big)$: in this case the general element of $\vert D\vert$ is smooth.
\end{enumerate}
\end{proposition}
\begin{proof}
See \cite{SD}, Proposition 2.6 and Corollary 3.2.
\end{proof}

Now we turn our attention to aCM line bundles on $F$ recalling the important results from \cite{Wa}. A first obvious fact is that for each  divisor $D$ on $F$, then $\cO_F(D)$ is aCM if and only if the same is true for $\cO_F(-D)$.

The main results from \cite{Wa} are summarized in the following statements.

\begin{lemma}
Let $F\subseteq\p3$ be a smooth quartic surface.

For each effective divisor $D$ on $F$ with $D^2\ge0$ the following assertions hold
\begin{enumerate}
\item If $D\ne0$, then $hD\ge3$.
\item If $hD=3$, then $\cO_F(D)$ is globally generated.
\end{enumerate}
\end{lemma}
\begin{proof}
See \cite{Wa}, Lemma 2.1 and Corollary 2.1.
\end{proof}

\begin{proposition}
\label{pWa}
Let $F\subseteq\p3$ be a smooth quartic surface.

For each effective divisor $D$ on $F$ with $D\ne0$, then $\cO_F(D)$ is initialized and aCM if and only if one of the following cases occurs.
\begin{enumerate}[(a)]
\item $D^2=-2$ and $1\le hD\le3$.
\item $D^2=0$ and $3\le hD\le4$.
\item $D^2=2$ and $hD=5$.
\item $D^2=4$, $hD=6$ and $h^0\big(F,\cO_F(D-h)\big)=h^0\big(F,\cO_F(2h-D)\big)=0$.
\end{enumerate}
\end{proposition}
\begin{proof}
See \cite{Wa}, Theorem 1.1.
\end{proof}

Notice that the general smooth quartic surface $F\subseteq\p3$ does not support initialized aCM line bundles besides $\cO_F$. Indeed, in this case, $\Pic(F)$ is generated by $\cO_F(h)$, due to Noether--Lefschetz theorem. In particular, for each line bundle $\cO_F(D)$ on $F$ one has that both $Dh$ and $D^2$ are positive multiples of $4$.

We analyze the above effective divisors defining initialized and aCM line bundles on $F$. Recall that a curve $D\subseteq\p3$ is called aCM if $h^1\big(\p3,\cI_{D\vert\p3}(th)\big)=0$ for each $t\in\bZ$ (see \cite{Mi}, Lemma 1.2.3). A curve $D\subseteq\p3$ is projectively normal if and only if it is aCM and smooth.

The first important fact is the following (see Proposition 2.14 of \cite{C--K--M2}). 

\begin{lemma}
\label{laCM}
Let $F\subseteq\p3$ be a smooth quartic surface.

For each effective divisor $D$ on $F$, then $\cO_F(D)$ is aCM if and only if the curve $D$ is aCM in $\p3$.
\end{lemma}
\begin{proof}
We have the exact sequence
\begin{equation}
\label{seqInclusion}
0\longrightarrow \cI_F\longrightarrow \cI_D\longrightarrow \cI_{D\vert F}\longrightarrow 0.
\end{equation}
Since $\cI_F\cong\cO_{\p3}(-4)$ and $\cI_{D\vert F}\cong\cO_F(-D)$, then we deduce that $h^1\big(\p3,\cI_D(t)\big)=0$ if and only if $h^1\big(F, \cO_F(D-th)\big)=h^1\big(F, \cI_{D\vert F}(th)\big)=0$.
\end{proof}

In what follows we will deal with some of the aCM line bundles of Proposition \ref{pWa}. We start from examining the initialized aCM line bundles $\cO_F(D)$ with $D^2=4$ and $Dh=6$. The divisor $D$ is aCM due to Proposition \ref{pWa} and Lemma \ref{laCM}. Proposition 6.2 of \cite{Bea} shows that such line bundles have a minimal free resolution on $\p3$ of the form
\begin{equation}
\label{seqDeterminantal}
0\longrightarrow\cO_{\p3}(-1)^{\oplus4}\mapright{\varphi_0}\cO_{\p3}^{\oplus4}\longrightarrow\cO_F(D)\longrightarrow0,
\end{equation}
where the matrix $A$ of $\varphi_0$ satisfies $f=\det(A)$. In particular $\cO_F(D)$ is globally generated.

\begin{lemma}
\label{lSextic}
Let $F\subseteq\p3$ be a smooth quartic surface.

If $D$ is an effective divisor on $F$ with $h^0\big(F,\cO_F(D-h)\big)=h^0\big(F,\cO_F(2h-D)\big)=0$, $D^2=4$, $Dh=6$, then $\cO_F(D)$ is globally generated. The general element of $\vert D\vert$ is an irreducible, projectively normal, sextic curve of genus $3$.
\end{lemma}
\begin{proof}
The proof follows immediately from the above discussion and Proposition \ref{SD}.
\end{proof}

Take an effective divisor with $D^2=-2$ and $1\le Dh\le 3$. It is obvious that if $Dh=1$, then $D$ is an line. 

Let us consider the two remaining cases. In these cases the Hilbert polynomial of $D$ is $Dh t+1$, because $p_a(D)=0$ (see the second Equality \eqref{h^1}). It follows that if $Dh=2$, then $D$ is contained in a plane, thus it is a possibly reducible or non--reduced conic.

Let $Dh=3$. Recall that for each subscheme $X\subseteq F$ we have the exact sequence
\begin{equation}
\label{seqStandard}
0\longrightarrow\cI_{X\vert F}\longrightarrow\cO_F\longrightarrow\cO_X\longrightarrow0.
\end{equation}
The cohomology of the above sequence for $X=D$, the fact that $\cO_F(D)$ is aCM and Lemma \ref{laCM} above, yield that $D$ is not contained in any plane (i.e. $h^0\big(F,\cO_F(h-D)\big)=h^0\big(F,\cI_{D\vert F}(h)\big)=0$).

The same argument shows that $D$ is contained in the intersection of three linearly independent quadrics. Since $D$ is an aCM cubic with $p_a(D)=0$, it follows that it is actually the intersection of these three quadrics. In particular the linear system $\vert 2h-D\vert$ has no fixed components, thus $\cO_F(2h-D)$ is globally generated and initialized. These remarks have the following helpful consequence.

\begin{lemma}
\label{lQuintic}
Let $F\subseteq\p3$ be a smooth quartic surface.

If $D$ is an effective divisor on $F$ with $D^2=2$, $Dh=5$, then $\cO_F(D)$ is globally generated. The general element of $\vert D\vert$ is an irreducible, projectively normal, quintic curve of genus $2$.
\end{lemma}
\begin{proof}
From the second Equality \eqref{h^1}, we know that $p_a(D)=2$. The divisor $D$ is aCM (see Proposition \ref{pWa} and Lemma \ref{laCM}), $h^0\big(F,\cO_F(D-h)\big)=0$ and its Hilbert polynomial is $5t-1$, thus $\vert 2h-D\vert\ne\emptyset$. We have $(2h-D)^2=-2$ and $(2h-D)h=3$, "thus each effective element in $\vert 2h-D\vert$ is aCM (see again Proposition \ref{pWa} and Lemma \ref{laCM}). Due to the above discussion $\cO_F(D)=\cO_F(2h-(2h-D))$ is globally generated. Bertini theorem implies that the general element of $\vert D\vert$ is irreducible and smooth, hence it is projectively normal.
\end{proof}

We now look at the case $D^2=0$ and $Dh=4$. The divisor $D$ is aCM (see Proposition \ref{pWa} and Lemma \ref{laCM}), $h^0\big(F,\cO_F(D-h)\big)=0$ and its Hilbert polynomial is $4t$, thus $D$ is contained in the base locus of a pencil of quadrics. If equality holds, then $\vert 2h-D\vert$ has no fixed components, whence $\cO_F(2h-D)$ is globally generated.

Assume that $D$ is not the base locus of the pencil of quadrics containing it, i.e. $\cO_F(2h-D)$ is not globally generated. By degree reasons we necessarily have that all the quadrics of the pencil split in a fixed plane $H$ plus another plane varying in a pencil. In particular the base locus of this pencil of quadrics is $H$ and a simple line $L$. If $D\subseteq H$, then $D=F\cap H$ because they are both quartic curves, hence $D^2=4$, a contradiction.

It follows that $D=C+L$ for a suitable cubic curve $C\subseteq H$. In particular $L\subseteq F$ and $p_a(C)=1$, hence $L^2=-2$ and $C^2=0$. Thus the equality $0=(C+L)^2$ forces $CL=1$.

Notice that $H\cap F=C\cup L'$ for a second suitable line. Planes through $L'$ cut out on $F$ residually to $L'$ the pencil of plane cubics $\vert C\vert$. Since $\vert C\vert+L\subseteq\vert D\vert$ and they have the same dimension, then they coincide: in particular $\cO_F(D)$ is not globally generated. 

The equalities $(2h-D)^2=0$, $(2h-D)h=4$ and  $D=2h-(2h-D)$ imply that all the arguments above are still true also if we start from $\cO_F(D)$ instead of $\cO_F(2h-D)$.

The above discussion essentially proves the following lemma: we leave the very easy details to the reader.

\begin{lemma}
\label{lQuartic}
Let $F\subseteq\p3$ be a smooth quartic surface.

If $D$ is an effective divisor on $F$ with $D^2=0$, $Dh=4$, then $\cO_F(D)$ is globally generated if and only if the same is true for $\cO_F(2h-D)$. Moreover, all the elements in $\vert D\vert$ and $\vert 2h-D\vert$ are projectively normal and
\begin{enumerate}
\item if $\cO_F(D)$ is globally generated, then each element of $\vert D\vert$ and $\vert 2h-D\vert$ is intersection of quadrics and the general one is an irreducible smooth elliptic quartic curve;
\item if $\cO_F(D)$ is not globally generated, then no elements of $\vert D\vert$ and $\vert 2h-D\vert$ are intersection of quadrics. All the elements are the union of a fixed line and a plane cubic meeting the line at exactly one point.
\end{enumerate}
\end{lemma}

\section{The exotic bundles}
\label{sExotic}
In this section we will examine briefly  but in details the case $ n=8$ which is particularly interesting. 

In this case $\pf(\Phi)=\det(A)$. In particular $f=\det(A)$ is the determinant of a $4\times4$ matrix with linear entries. The cokernel of the induced morphism $\varphi_0\colon\cO_{\p3}(-1)^{\oplus4}\to\cO_{\p3}^{\oplus4}$ is a line bundle $\cO_F(D)$ and Sequence \eqref{seqDeterminantal} is a minimal free resolution of the aCM line bundle $\cO_F(D)$ with $D^2=4$, $Dh=6$. 

Grothendieck duality yields the exact sequence
\begin{equation}
\label{seqDeterminantalDual}
0\longrightarrow\cO_{\p3}(-2)^{\oplus4}\mapright{-{}^t\varphi_0(-2)}\cO_{\p3}(-1)^{\oplus4}\longrightarrow\cO_F(2h-D)\longrightarrow0.
\end{equation}
Thanks to Corollary 1.8 of \cite{Bea}, we know that also  $\cO_F(2h-D)$ is aCM and the above sequence is its minimal free resolution. Moreover, the resolution above implies $h^0\big(F,\cO_F(2h-D)\big)=0$. 

We have the following commutative diagram
\begin{equation*}
\label{CDExotic}
\begin{CD}
0@>>>\cO_F(-1)^{\oplus4}@>\varphi_0>>   \cO_F^{\oplus4}@>>> \cO_F(D)@>>>0\\
@.@VV\psi V  @V\theta VV \\
0@>>> \cO_F(-2)^{\oplus4}\oplus\cO_F(-1)^{\oplus4}@>\varphi>>   \cO_F^{\oplus4}\oplus\cO_F(-1)^{\oplus4}@>>>\cE@>>>0\\
@.@VV{}^t\theta(-2)V  @V{}^t\psi(-1) VV \\
0@>>> \cO_F(-2)^{\oplus4}@>-{}^t\varphi_0(-2)>>   \cO_F(-1)^{\oplus4}@>>>\cO_F(2h-D)@>>>0\\
\end{CD}
\end{equation*}
where $\varphi$ has matrix $\Phi$ as in Equality \eqref{matrixPhi}, and $\psi$, $\theta$ have respective matrices
$$
\left(\begin{array}{c}
0\\
I
\end{array}
\right),\qquad \left(\begin{array}{c}
I\\
0
\end{array}
\right),
$$
$I$ being the identity of order $4$. Since the two first columns of the above diagram are trivially short exact sequences, it follows the existence of an exact sequence of the form 
\begin{equation}
\label{seqSplitExotic}
0\longrightarrow\cO_F(D)\longrightarrow\cE\longrightarrow\cO_F(2h-D)\longrightarrow0,\qquad Dh=6,\ p_a(D)=3.
\end{equation}
The above discussion proves part of the following statement

\begin{theorem}
\label{t2Exotic}
Let $F\subseteq\p3$ be a smooth quartic surface.

There exists a bundle $\cE$ of rank $2$ on $F$ whose minimal free resolution on $\p3$ is Resolution \eqref{resGeneral} with $n=8$ if and only if the equation of $F$ can be expressed as the determinant of a $4\times4$ matrix with linear entries. 

A rank $2$ bundle $\cE$ on $F$  is as above if and only if $\cE$ fits into Sequence \eqref{seqSplitExotic} where $D$ is an irreducible smooth sextic curve of genus $3$ not lying on a quadric and such that $h^0\big(F,\cO_F(D-h)\big)=0$. 

Indecomposable bundles fitting in Sequence \eqref{seqSplitExotic} form a family parameterized by $\p5$. Decomposable bundles are exactly the direct sums $\cO_F(D)\oplus\cO_F(2h-D)$ where $D$ is a curve as above. 
\end{theorem}
\begin{proof}
We proved above that the existence of Resolution \eqref{resGeneral} with $n=8$ implies the following facts:
\begin{itemize}
\item the equation of $F$ can be expressed as the determinant of a $4\times4$ matrix with linear entries;
\item a divisor $D$ as above exists on $F$;
\item Sequence \eqref{seqSplitExotic} exists.
\end{itemize}

Assume that the equation of $F$ can be expressed as the determinant of a $4\times4$ matrix with linear entries. Then we can construct Sequences \eqref{seqDeterminantal} and \eqref{seqDeterminantalDual}, defining line bundles $\cO_F(D)$ and $\cO_F(2h-D)$. Thus $\cO_F(D)\oplus\cO_F(2h-D)$ is a bundle of rank $2$ whose minimal free resolution on $\p3$ is Resolution \eqref{resGeneral} with $n=8$. This completes the proof of the first equivalence in the statement.

Assume the existence of $D$ on $F$ as in the statement and let $\cE$ any sheaf defined by Sequence \eqref{seqSplitExotic}. 
We have $D^2=4$, $hD=6$, $h^0\big(F,\cO_F(D-h)\big)=h^0\big(F,\cO_F(2h-D)\big)=0$ by hypothesis, thus Proposition \ref{pWa} implies that $\cO_F(D)$ is aCM and initialized. 

Moreover it is easy to check that $(3h-D)^2=4$, $(3h-D)h=6$: the two above vanishings give analogous vanishings with the divisor $2h-D$ instead of $D$. Thus $\cO_F(3h-D)$ is aCM and initialized as well, again due to Proposition \ref{pWa}, whence $\cO_F(2h-D)$ is aCM. It is an immediate consequence of these remarks that $\cE$ is an initialized, aCM bundle of rank $2$. Moreover $c_1(\cE)=2h$ and $c_2(\cE)=8$.

We know that $\cO_F(D)$ is aCM, thus we can write the following commutative diagram
\begin{equation*}
\label{CDGG}
\begin{CD}
H^0\big(F,\cE\big)\otimes\cO_F@>>>H^0\big(F,\cO_F(2h-D)\big)\otimes\cO_F@>>>0\\
@VVV @VVV\\
\cE@>>>\cO_F(2h-D)@>>>0.\\
\end{CD}
\end{equation*}
Since $\cO_F(2h-D)$ is not globally generated (because $h^0\big(F,\cO_F(2h-D)\big)=0$), the same is true for $\cE$. In particular $\cE$ fits in Resolution \eqref{resGeneral} with $n$ either $6$ or $8$. 

Assume that the first case holds. 
We can construct a commutative diagram
\begin{equation}
\label{CDAbsurd}
\begin{CD}
0@>>>\cO_{\p3}(-1)^{\oplus4}@>\varphi_0>>   \cO_{\p3}^{\oplus4}@>>> \cO_F(D)@>>>0\\
@.@. @. @VVV \\
0@>>> \cO_{\p3}(-2)^{\oplus4}\oplus\cO_{\p3}(-1)^{\oplus2}@>\overline{\varphi}>>   \cO_{\p3}^{\oplus4}\oplus\cO_{\p3}(-1)^{\oplus2}@>>>\cE@>>>0
\end{CD}
\end{equation}
where the matrices of $\varphi_0$ and $\overline{\varphi}$ are the already defined $A$ and
$$
\left(\begin{array}{cc}
\overline{B}&\overline{A}\\
-{}^t\overline{A}&0
\end{array}
\right).
$$
Here $\overline{B}$ and $\overline{A}$ are respectively a $4\times4$ matrix with quadratic entries and  a $4\times2$ matrix with linear entries. Let $S:=k[x_0,x_1,x_2,x_3]$: taking global sections we have an induced diagram
$$
\begin{CD}
S^{\oplus4}@>>> H^0_*\big(F,\cO_F(D)\big)\\
@. @| \\
S^{\oplus4}\oplus S(-1)^{\oplus2}@>>>H^0_*\big(F,\cE\big)@>>>0
\end{CD}
$$
where the bottom row is exact because it comes from a minimal free resolution of $\cE$. Thus there is a map 
$\Theta\colon S^{\oplus4}\to S^{\oplus4}\oplus S(-1)^{\oplus2}$ making the above diagram commutative. By construction, in degree $0$ the maps in the above diagram are isomorphisms. Thus the degree $0$ component of $\Theta$ must have maximal rank. Up to a proper choice of the bases, we can assume that the matrix of $\Theta$ is
$$
\left(\begin{array}{c}
I\\
0
\end{array}
\right),
$$
where $I$ is the identity of order $4$. 

By sheafifying we thus obtain a morphism $\theta\colon \cO_{\p3}^{\oplus4}\to \cO_{\p3}^{\oplus4}\oplus\cO_{\p3}(-1)^{\oplus2}$ making Diagram \eqref{CDAbsurd} commutative. It follows the existence of another morphism $\eta\colon\cO_F(-1)^{\oplus4}\to\cO_F(-2)^{\oplus4}\oplus\cO_F(-1)^{\oplus2}$ completing Diagram \eqref{CDAbsurd} as a commutative diagram. By degree reasons its matrix is
$$
\left(\begin{array}{c}
0\\
C
\end{array}
\right)
$$
where $C$ is a $2\times4$ matrix of constants. The commutativity of the last square of the diagram would imply $A=\overline{A}C$ which is a contradiction, because at each point of ${\p3}$ the rank of the matrix on the right is at most $2$. We conclude that $\cE$ fits into Resolution \eqref{resGeneral} with $n=8$. Thus also the proof of the second equivalence in the statement is complete.

Assume that $\cE$ is decomposable: then $\cE\cong \cO_F(A)\oplus\cO_F(2h-A)$ (because $c_1(\cE)=2h$), where $\cO_F(A)$ and $\cO_F(2h-A)$ are aCM (because summands of an aCM bundle), $\cO_F(A)$ is initialized and either $\cO_F(2h-A)$ is also initialized, or it has no sections (because summands of an initialized bundle). 

Proposition \ref{pWa} and equality $(2h-A)A=c_2(\cE)=8$  imply that either $A^2=-2$, $hA=3$, $(2h-A)^2=2$, $h(2h-A)=5$, or $A^2=(2h-A)^2=0$ and $hA=h(2h-A)=4$, or $A^2=2$, $hA=5$, $(2h-A)^2=-2$, $h(2h-A)=3$, or finally $A^2=4$, $hA=6$, $(2h-A)^2=-4$, $h(2h-A)=2$. In the first three cases $(A-D)h$ and $(2h-A-D)h$ are both negative, thus $h^0\big(F,\cO_F(A-D)\big)=h^0\big(F,\cO_F(2h-A-D)\big)$. In particular each morphism $\cO_F(D)\to\cO_F(A)\oplus\cO_F(2h-A)$ must be zero.

Thus only the last case is possible. Since $(2h-A-D)h=-4$ we deduce that an injective morphism $\cO_F(D)\to\cO_F(A)\oplus\cO_F(2h-A)$ induces an injective morphism $\cO_F(D)\to\cO_F(A)$, i.e. a non--zero section in  $H^0\big(F,\cO_F(D-A)\big)$. Thus $\cO_F(D)\cong\cO_F(A)$ because $(D-A)h=0$, hence Sequence \eqref{seqSplitExotic} splits. We conclude that there are extensions as above with an indecomposable sheaf $\cE$ in the middle.

Finally, indecomposable bundles fitting in Sequence \eqref{seqSplitExotic} are parameterized by a projective space of dimension $h^1\big(F,\cO_F(2D-2h)\big)-1$. We have $h^2\big(F,\cO_F(2D-2h)\big)=h^0\big(F,\cO_F(2h-2D)\big)=0$, because $(2h-2D)h=-4$. We have $(2D-2h)D=-4$ which implies that each effective divisor splits as the sum of $D$ and an effective divisor in $\vert D-2h\vert$. The inequality $(D-2h)h=-2$ thus gives a contradiction. We conclude that $h^0\big(F,\cO_F(2D-2h)\big)=0$, too. 

Equality \eqref{RRgeneral} gives $\chi(\cO_F(2D-2h))=2+2(D-h)^2=-6$, thus $h^1\big(F,\cO_F(2D-2h)\big)=6$. In particular, the non--trivial Sequences \eqref{seqSplitExotic} are parameterized by $\p5$.
\end{proof}

\begin{remark}
\label{0Exotic}
The cohomology of Sequence \eqref{seqSplitExotic} gives $h^0\big(F,\cE(-D)\big)=1$ and $h^0\big(F,\cE\big)=h^0\big(F,\cO_F(D)\big)$. Hence the natural multiplication morphism 
$$
H^0\big(F,\cE(-D)\big)\otimes H^0\big(F,\cO_F(D)\big)\longrightarrow H^0\big(F,\cE\big)
$$
is an isomorphism. We deduce that the zero--locus of each non--zero section of $\cE$ is a divisor in $\vert D\vert$ and conversely.
\end{remark}

\begin{remark}
\label{rTwiceExotic}
Let $\cE$ fit into Resolution \eqref{resGeneral} with $n=8$.

Then $\cE(h)$ is globally generated, thus the same is true for $\cO_F(3h-D)$. We conclude that the general element in $\vert 3h-D\vert$ is a smooth irreducible curve. As pointed out in the proof above  $(3h-D)^2=4$, $(3h-D)h=6$, thus the genus of the general element in $\vert 3h-D\vert$ is $3$. Moreover
\begin{gather*}
h^0\big(F,\cO_F(2h-(3h-D))\big)=h^0\big(F,\cO_F(D-h)\big)=0,\\
h^0\big(F,\cO_F((3h-D)-h)\big)=h^0\big(F,\cO_F(2h-D)\big)=0,
\end{gather*}
thus Proposition \ref{pWa}, (d) implies that $\cO_F(3h-D)$ is aCM. Finally $D(3h-D)=14\ne D^2$, hence $\vert D\vert\ne\vert 3h-D\vert$. 

Replacing $D$ with $3h-D$ in the construction above, we deduce the existence of a second family of indecomposable, initialized aCM bundles $\overline{\cE}$ on $F$ with $c_1(\overline{\cE})=2h$ and $c_2(\overline{\cE})=8$ fitting in a different Resolution \eqref{resGeneral} again with $n=8$.

The two families are disjoint because the bundles $\cE$ and $\overline{\cE}$ have sections vanishing on divisors lying in disjoint linear systems.
\end{remark}

\section{Construction of bundles}
\label{sConstruction}
Let $\cE$ be a rank $2$ vector bundle on $F$ and $s\in H^0\big(F,\cE\big)$. If its zero--locus $E:=(s)_0\subseteq F$ has pure codimension $2$, the corresponding Koszul complex  gives the following exact sequence 
\begin{equation}
  \label{seqIdeal}
  0\longrightarrow \cO_F\longrightarrow\cE\longrightarrow \cI_{E\vert F}(c_1(\cE))\longrightarrow 0.
\end{equation}
The degree of $E$ is $c_2(\cE)$. We now recall how to revert the  above construction. To this purpose we give the following well--known definition.

\begin{definition}
Let $F$ be a smooth projective irreducible surface and let $\cG$ be a coherent sheaf on $F$.

We say that a locally complete intersection subscheme $E\subseteq F$ of dimension zero is Cayley--Bacharach (CB for short) with respect to $\cG$ if, for each $E'\subseteq E$ of degree $\deg(E)-1$, the natural morphism $H^0\big(F,\cI_{E\vert F}\otimes\cG\big)\to H^0\big(F,\cI_{E'\vert F}\otimes\cG\big)$ is an isomorphism.
\end{definition}

For the following result see Theorem 5.1.1 in \cite{H--L}.

\begin{theorem}
  \label{tCB}
 Let $F$ be an integral smooth projective irreducible surface, $E\subseteq F$ a locally complete intersection subscheme of dimension $0$. 
 
Then there exists a vector bundle $\cE$ of rank $2$ on $F$ with $\det(\cE)={\mathcal L}$ and having a section $s$ such that $E=(s)_0$ if and only if $E$ is CB with respect to $\omega_F\otimes{\mathcal L}$. 

If such an $\cE$ exists, then it is unique up to isomorphism if $h^1\big(F,{\mathcal L}^\vee\big)= 0$.
\end{theorem}

In the remaing part of this section, making use of the above theorem, we will investigate which are the conditions that a zero--dimensional scheme $E\subseteq F$ must satisfy in order to be the zero--locus of an initialized aCM bundle $\cE$ of rank $2$ on $F$ with $c_1(\cE)=\cO_F(2h)$. 

Let $E\subseteq F$ be a zero--dimensional subscheme. Its Hilbert polynomial is $\deg(E)$ and let $I_E\subseteq S$ be its homogeneous ideal. 

For each general linear form $\ell\in S_E:=S/I_E$, the ring $R:= S_E/\ell S_E$ is local, graded and Artin: $R$ is called {\sl the Artin reduction of $S_E$}\/ and it is naturally graded. The maximum integer $\sigma>0$ such that the component $R_\sigma$ of degree $\sigma$ is non--zero is called {\sl the socle degree of $R$}\/: it depends only on $E$ and not on $\ell$. We say that $R$ is {\sl Gorenstein} if $R_\sigma$ has dimension $1$ as vector space over $k$.

We say that $E$ is {\sl arithmetically Gorenstein} ({\sl aG} for short) if its Artin reduction is Gorenstein: the aG property does not depend on $\ell$ as well.

For the following helpful result see \cite{Kr}, Theorem 1.1.

\begin{theorem}
\label{tDGO}
Let $E\subseteq\p n$ be a zero--dimensional scheme and let $\sigma$ be the socle degree of the Artin reduction of its homogeneous coordinate ring. 

Then $E$ is aG if and only if the two following conditions hold.
\begin{enumerate}
\item $\dim_k((S_E)_t)+\dim_k((S_E)_{\sigma-1-t})=\deg(E)$, for each $t\in \bZ$.
\item $E$ is CB with respect to $\cO_{\p n}(\sigma-1)$.
\end{enumerate}
\end{theorem}

Notice that if $E$ is properly contained in an aCM scheme $F\subseteq \p n$, then  the twisted cohomology of Sequence \eqref{seqInclusion}
allows us to replace $\cO_{\p n}(\sigma-1)$ with $\cO_F((\sigma-1)h)$ in Theorem \ref{tDGO} above. Thus if $E$ is aG, then it is also CB with respect to $\cO_F((\sigma-1)h)$, hence we can construct a vector bundle $\cE$ of rank $2$ fitting into Sequence \eqref{seqIdeal} with $\mathcal L=\cO_F((\sigma-1)h)$ thanks to Theorem \ref{tCB}. 

We can now state the main result of this section.

\begin{proposition}
\label{pExistence}
Let $F\subseteq\p3$ be a smooth quartic surface.

\begin{enumerate}
\item Let $\cE$ be an initialized, aCM, bundle of rank $2$ on $F$ with $c_1(\cE)=\cO_F(2h)$, $c_2(\cE)=8$, such that the zero--locus of its  general section is zero--dimensional: then $E$ is an aG scheme of degree $8$ not contained in any plane.

\item Conversely, each aG scheme $E\subseteq F$ of degree $8$ not contained in any plane arises as the zero--locus of some section of a unique initialized, aCM, bundle $\cE$ of rank $2$ on $F$ with $c_1(\cE)=\cO_F(2h)$, $c_2(\cE)=8$.
\end{enumerate}
\end{proposition}
\begin{proof}
Assume that $\cE$ is as in (1). Trivially $\deg(E)=8$ and we have to prove that $E$ is aG. Theorem \ref{tCB} implies that $E$ is CB with respect to $\cO_F(2h)$. We thus have only to check that $\dim_k((S_E)_t)+\dim_k((S_E)_{2-t})=8$ for each $t\in\bZ$. We obviously have $\dim_k((S_E)_0)=1$ and $\dim_k((S_E)_t)=0$ for $t\le-1$.

The cohomology of Sequence \eqref{seqIdeal} implies $h^1\big(F,\cI_{E\vert F}(th)\big)=h^1\big(F,\cE((t-2)h)\big)=0$ when $t\ge3$ because $h^2\big(F,\cO_F(t-2)\big)=0$ in the same range. The cohomology of Sequence \eqref{seqStandard} with $X=E$ thus gives
$$
\dim_k((S_E)_t)=h^0\big(F,\cO_E(th)\big)-h^1\big(F,\cI_{E\vert F}(th)\big)=8,\qquad t\ge3.
$$

When $t=2$, the cohomology of Sequence \eqref{seqIdeal} gives the exact sequence 
\begin{equation}
\label{seqCohomology}
\begin{aligned}
0\longrightarrow H^1\big(F,\cE\big)\longrightarrow H^1\big(F,\cI_{E\vert F}(2h)\big)&\longrightarrow H^2\big(F,\cO_F\big)\longrightarrow \\
\longrightarrow H^2\big(F,\cE\big)&\longrightarrow H^2\big(F,\cI_{E\vert F}(2h)\big)\longrightarrow 0.
\end{aligned}
\end{equation}
We have $h^1\big(F,\cE\big)=0$, $h^2\big(F,\cO_F\big)=1$ and $h^2\big(F,\cE\big)=h^0\big(F,\cE(-2h)\big)=0$. Thus $h^1\big(F,\cI_{E\vert F}(2h)\big)=1$ whence $\dim_k((S_E)_2)=7$. When $t=1$, a similar argument yields $h^1\big(F,\cI_{E\vert F}(h)\big)=4$, thus $\dim_k((S_E)_1)=4$: in particular $E$ is not contained in any plane. Thus assertion (1) is completely proved.

Conversely let $E$ be aG of degree $8$ on $F$ not contained in any plane. The Artin reduction $R$ (see above) of $S_E$ is a local graded, Gorenstein ring. It is well known that the possible Hilbert functions for $R$ must be symmetric. Thus the possible cases are:
\begin{equation}
\label{List}
(1,6,1),\qquad (1,3,3,1),\qquad(1,2,2,2,1),\qquad(1,1,1,1,1,1,1,1).
\end{equation} 

In the two last cases we would have $\dim_k((S_E)_1)<4$, hence $E$ would be contained in a plane. The first case cannot hold as well: indeed it would imply $\dim_k((S_E)_1)\ge7$. 

Thus we know that only the second case is possible. In particular the socle degree of the Artin reduction of $S_E$ is $3$. It follows that $E$ is CB with respect to $\cO_F(2h)$ and we can construct a rank $2$ bundle $\cE$ fitting into Sequence \eqref{seqIdeal} with $c_1(\cE)=2h$ and $c_2(\cE)=8$ having a section vanishing exactly along $E$ (Theorem \ref{tCB}). We have to check that $\cE$ is initialized and aCM.

From the Hilbert function of $R$, we also obtain  $\dim_k((S_E)_1)=4$, $\dim_k((S_E)_2)=7$ and $\dim_k((S_E)_t)=8$ for $t\ge3$. 
Consequently, from the cohomology of Sequence \eqref{seqStandard} for $X=E$,
\begin{equation}
\label{vanishingIdeal}
h^1\big(F,\cI_{E\vert F}(th)\big)=\left\lbrace\begin{array}{ll}
0 \quad&\text{if $t\ge3$,} \\
1 &\text{if $t=2$,}\\
4 &\text{if $t=1$.}
\end{array}\right.
\end{equation}

The cohomology of Sequence \eqref{seqIdeal} twisted by $\cO_F(-h)$ and the fact that $E$ is not contained in any plane implies that $h^0\big(F,\cE(-h)\big)=0$. Moreover each scheme of degree $8$ and dimension $0$ is trivially contained in at least two quadrics, thus the cohomology of the same sequence also gives $h^0\big(F,\cE\big)\ge2$. We conclude that $\cE$ is initialized.

Again the same sequence twisted by $\cO_F(th)$ and the first Equality \eqref{vanishingIdeal} imply $h^1\big(F,\cE(th)\big)=0$ for $t\ge1$. Looking at Sequence \eqref{seqCohomology}, we deduce $h^1\big(F,\cE\big)=0$, due to the second Equality \eqref{vanishingIdeal} and to the vanishing $h^2\big(F,\cE\big)=h^0\big(F,\cE(-2h)\big)=0$. Using the third Equality \eqref{vanishingIdeal} and $h^2\big(F,\cE(-h)\big)=h^0\big(F,\cE(-h)\big)=0$, a similar argument implies also $h^1\big(F,\cE(-h)\big)=0$. 

By Serre duality we obtain also the vanishing $h^1\big(F,\cE(th)\big)=h^1\big(F,\cE((-t-2)h)\big)=0$ for $t\le -2$. The proof is thus complete.
\end{proof}

\section{Construction of bundles with $n=4$}
\label{sConstruction2gg}
In this section we will show how to modify the methods in \cite{C--K--M1} in order to prove the existence of initialized aCM bundles of rank $2$ fitting into Resolution \eqref{resGeneral} with $n=4$. 

We take a general element $C\in \vert\cO_F(2h)\vert$. The curve $C$ can be assumed of bidegree $(4,4)$ on a smooth quadric surface, it has degree $8$ and genus $g:=9$. Notice that for each divisor $\Gamma$ on $C$ we have the exact sequence 
\begin{equation}
\label{seqCurve}
0\longrightarrow \cO_F(-2h)\longrightarrow \cI_{\Gamma\vert F}\longrightarrow \cO_C(-\Gamma)\longrightarrow0,
\end{equation}
because $ \cI_{C\vert F}\cong\cO_F(-2h)$.

\begin{lemma}
\label{lPencil}
Let $C$ be as above. There is a divisor $E$ on $C$ such that
\begin{enumerate}
\item $\deg(E)=8$;
\item $h^0(C,\cO_C(E)\big)=2$;
\item Both $\cO_C(E)$, and $\omega_C\otimes\cO_C(-E)$ are globally generated.
\end{enumerate}
\end{lemma}
\begin{proof}
Recall that the linear series of projective dimension at least $r$ and degree $\delta$ on $C$ are parameterized by a union of projective varieties usually denoted by $W^r_\delta(C)$: in what follows we will omit $C$ because it is fixed.

We have that $g-8+1\ge0$, thus no component of $W^1_8$ is entirely contained in $W^2_8$ (see \cite{A--C--G--H}, Theorem IV.3.5). In particular $W^1_8\setminus W^2_8$ is open and dense in each component of $W^1_8$.

Let us consider the map $C\times W^1_{7}\to W^1_8$ defined by $(P,\mathcal L)\mapsto \mathcal L\otimes\cO_C(P)$. Trivially the subset in $W^1_8$ of linear series with a base point is contained in the image $X$ of the above map.  Notice that Lemma IV.3.3 of \cite{A--C--G--H} implies that each component of $W^1_8$ has dimension at least the Brill--Noether number $\varrho(9,8,1)=5$.

The curve $C$ is a curve of bidegree $(4,4)$ on a smooth quadric. Hence Exercises IV. D--5 and IV.F--2 of \cite{A--C--G--H} imply that $C$ is neither hyperelliptic, nor trigonal, nor bielliptic. Moreover, it is not isomorphic to a smooth plane quintic, because its genus is $9$. 

Since $2\le g-2=7$, it follows from Theorems IV.5.1 and IV.5.2 of \cite{A--C--G--H} that each component of $W^1_{7}$ has dimension $3$. Thus each component of $X$ has dimension at most $4$, which is smaller than $\varrho(9,8,1)=5$.
In particular $W^1_8\setminus X$ is open and dense in each component of $W^1_8$.

Finally, Riemann--Roch theorem on $C$ implies that the residuation $i$ with respect to $\omega_C$ is an involution $i\colon W^1_8\to W^{1}_{8}$ because $\deg(\omega_C(-E))=\deg(E)$. In particular $i^{-1}(W^1_8\setminus X)$ is open and dense in each component of $W^1_8$.

It suffices to take a linear system $\mathcal L$ in the intersection 
$$
(W^1_8\setminus X)\cap i^{-1}(W^1_8\setminus X)\cap (W^1_8\setminus W^2_8)
$$
which is certainly non--empty due to the above discussion. Each $E\in\vert\mathcal L\vert$ is a divisor with the desired properties.
\end{proof}

We can now state and prove the main result of this section.

\begin{theorem}
\label{t2gg}
Let $F\subseteq\p3$ be a smooth quartic surface.

There exist bundles $\cE$ of rank $2$ on $F$ whose minimal free resolution on $\p3$ is Resolution \eqref{resGeneral} with $n=4$. 

The general bundle fitting in Resolution \eqref{resGeneral} with $n=4$ is indecomposable. Decomposable bundles are exactly the direct sums $\cO_F(D)\oplus\cO_F(2h-D)$ where $D$ is a smooth elliptic quartic curve on $F$.
\end{theorem}
\begin{proof}
Take $E\subseteq C\subseteq F$ as in Lemma \ref{lPencil}. Let $H$ be a general hyperplane section of $C$. Since $h^0\big(C,\cO_C(H)\big)\ne h^0\big(C,\cO_C(E)\big)$ and $\deg(H)=\deg(E)$ we conclude that $h^0\big(C,\cO_C(H-E)\big)=h^0\big(C,\cO_C(E-H)\big)=0$. In particular the twisted cohomology of Sequence \eqref{seqCurve} with $\Gamma=E$ and Riemann--Roch theorem on $C$ imply
$$
h^0\big(F,\cI_{E\vert F}(th)\big)=\left\lbrace\begin{array}{ll}
0 \quad&\text{if $t=0,1$,} \\
3 \quad&\text{if $t=2$,} \\
12 &\text{if $t=3$.}
\end{array}\right.
$$
A very easy computation now shows that condition (1) in Theorem \ref{tDGO} is satisfied by $E$ with $\sigma=3$. In order to construct an aCM vector bundle from $E$ we only need to show that $E$ is CB with respect to $\cO_F(2h)$.

To this purpose we notice that twisting the Sequence \eqref{seqCurve} by $\cO_F(2h)$ and taking its cohomology, the isomorphism $\cO_F(2h)\otimes\cO_C\cong\omega_C$ implies that
$$
h^0\big(F,\cI_{\Gamma\vert F}(2h)\big)=9-\deg(\Gamma)+h^0\big(C,\cO_C(\Gamma)\big).
$$
For each subscheme $E'\subseteq E$ of degree $7$, Lemma \ref{lPencil} implies that $h^0\big(C,\cO_C(E')\big)=1$, because $\cO_C(E')$ is globally generated. Thus the above equality yields $h^0\big(F,\cI_{E'\vert F}(2h)\big)=3=h^0\big(F,\cI_{E\vert F}(2h)\big)$, i.e. $E$ is CB with respect to $\cO_F(2h)$. 

It follows that $E$ is aG of degree $8$, hence Proposition \ref{pExistence} yields the existence of an initialized, aCM bundle $\cE$ of rank $2$ on $F$ with $c_1(\cE)=2h$ and $c_2(\cE)=8$.

In particular the cohomology of Sequence \eqref{seqIdeal} twisted by $\cO_F(2h)$ gives $h^1\big(F,\cI_{E\vert F}(2h)\big)=h^1\big(F,\cE(2h)\big)=0$, hence we have the diagram with exact rows
\begin{equation}
\label{CD}
\begin{CD}
0@>>> \cO_F@>>>   H^0\big(F,\cI_{E\vert F}(2h)\big)\otimes \cO_F@>>> H^0\big(C,\omega_C(-E)\big)\otimes \cO_C@>>>0\\
@.@|  @VVV @VVV\\
0@>>> \cO_F@>>>   \cI_{E\vert F}(2h)@>>>\omega_C(-E)@>>>0.\\
\end{CD}
\end{equation}
Five's Lemma and the surjectivity of the vertical map on the right (following from Lemma \ref{lPencil}) imply that $\cI_{E\vert F}(2h)$ is globally generated.  

A similar argument applied to Sequence \eqref{seqInclusion} twisted by $\cO_{\p3}(2)$ yields that $\cI_E(2)$ is globally generated too, thus $E$ is the complete intersection of three quadrics. 

Again the same argument applied to Sequence \eqref{seqIdeal} and the vanishing $h^1\big(F,\cO_F\big)=0$ yield that  $\cE$ is globally generated too, thus its minimal free resolution is Resolution \eqref{resGeneral} with $n=4$. Thus the proof of the existence of the bundle $\cE$ is completed.

Assume that such a bundle $\cE$ is decomposable. We can write $\cE\cong\cO_F(D)\oplus\cO_F(2h-D)$ for some divisor $D$ on $F$. Moreover, we know that $(2h-D)D=c_2(\cE)=8$ and
\begin{equation}
\label{sumCohomology}
h^i\big(F,\cE\big)=h^i\big(F,\cO_F(D)\big)+h^i\big(F,\cO_F(2h-D)\big),\qquad i=0,1,2.
\end{equation}
Thus, both $\cO_F(D)$ and $\cO_F(2h-D)$ are globally generated and aCM. It follows that $h^0\big(F,\cO_F(D)\big)=h^0\big(F,\cO_F(2h-D)\big)\ge2$, hence the equality must hold, because $h^0\big(F,\cE\big)=4$. In particular $\cO_F(D)$ and $\cO_F(2h-D)$ are initialized and $D^2, (2h-D)^2\ge0$.

Due to case (d) of Proposition \ref{pWa}, $D^2=4$ is not possible, because $h^0\big(F,\cO_F(2h-D)\big)\ne0$. If $D^2=2$, then $hD=5$, hence $(2h-D)^2=-2$, contradicting its positivity proved above. Thus $D^2=(2h-D)^2=0$ and, consequently, $hD=(2h-D)h=4$, because $(2h-D)D=8$. Finally we also deduce that $D$ and $2h-D$ can be assumed to be smooth elliptic quartic curves (Proposition \ref{SD} and Bertini theorem). 

We now prove the existence of indecomposable, initialized aCM bundle $\cE$ of rank $2$ with $c_1(\cE)=\cO_F(2h)$ and $c_2(\cE)=8$ on $F$ whose minimal free resolution on $\p3$ is Resolution \eqref{resGeneral} with $n=4$. To this purpose we first notice that, if $F$ does not contain smooth elliptic quartic curves each such bundle (whose existence was already proved above) is automatically indecomposable, due to the above discussion.

On the other hand, assume that $F$ contains a smooth elliptic quartic curve $D$: thus both $\cO_F(D)$ and $\cO_F(2h-D)$ are globally generated (see Proposition \ref{SD} and Lemma \ref{lQuartic}). For each non--zero $t\in\bZ$ we have $\cO_F(th)\not\cong\cO_F(tD)$, because $h^2=4\ne 0=D^2$. Thus both $\cO_F(t(D-h))$ and $\cO_F(t(h-D))$ are not effective because $t(h-D)h=0$. Thus they are both initialized.

Equality \eqref{RRgeneral}  implies $\chi(\cO_F(2D-2h))=-6$, hence $h^1\big(F,\cO_F(2D-2h)\big)\ge6$.  
It follows the existence of many non trivial extensions
\begin{equation}
\label{seqSplitGG}
0\longrightarrow\cO_F(D)\longrightarrow\cE\longrightarrow\cO_F(2h-D)\longrightarrow0,\qquad Dh=4,\ p_a(D)=1.
\end{equation}

If $\cE$ is splitting, then the first part of the proof would show the existence of a  smooth elliptic quartic curve $A$ on $F$ such that $\cE\cong\cO_F(A)\oplus\cO_F(2h-A)$. Thus we should have a map $\cO_F(D)\to\cO_F(A)$, i.e. a section in $H^0\big(F,\cO_F(D-A)\big)$. If such a space is non--zero, then $\cO_F(D)\cong\cO_F(A)$ and the above sequence would split. Thus such a map is zero and we have a non--zero map $\cO_F(D)\to\cO_F(2h-A)$. It follows that $\cO_F(2h-A)\cong\cO_F(D)$ and $\cO_F(A)\cong\cO_F(2h-D)$. 

We deduce that there are extensions as above with an indecomposable sheaf $\cE$ in the middle. Recall that $\cO_F(D)$ and $\cO_F(2h-D)$ are globally generated, aCM and initialized. Taking the cohomology of Sequence \eqref{seqSplitGG} twisted by $\cO_F(th)$ is then very easy to check that $\cE$ is aCM and initialized as well. Moreover, the same argument used in Diagram \eqref{CD} and the vanishing $h^1\big(F,\cO_F(D)\big)=0$ also imply that $\cE$ is globally generated. We conclude that it fits in Resolution \eqref{resGeneral} with $n=4$. Finally it is immediate to check that $c_1(\cE)=\cO_F(2h)$ and $c_2(\cE)=8$.
\end{proof}

\begin{remark}
\label{rTwiceGG}
Let $\cE$ fit into Resolution \eqref{resGeneral} with $n=4$ and into Sequence \eqref{seqSplitGG} for some smooth elliptic quartic curve $D$. Then $\cO_F(2h-D)$ is globally generated too, hence there is a smooth elliptic quartic curve $\overline{D}\in \vert 2h-D\vert$. 

As in Remark \ref{rTwiceExotic}, $D(2h-D)=4\ne D^2$, thus $\vert D\vert\ne \vert 2h-D\vert$. We can construct a second family of indecomposable bundles $\overline{\cE}$ fitting into Resolution \eqref{resGeneral} with $n=4$. Moreover each such bundle fits into
$$
0\longrightarrow \cO_F(2h-D)\longrightarrow\overline{\cE} \longrightarrow\cO_F(D)\longrightarrow0.
$$

Assume the existence of an isomorphism $\cE\cong\overline{\cE}$ where $\cE$ fits into Sequence \eqref{seqSplitGG}. We would have by composition a map $\cO_F(D)\to\cO_F(D)$. If this map is zero, then the image of $\cO_F(D)$ inside $\overline{\cE}$ would be contained in the image of $\cO_F(2h-D)$ in the above sequence, hence we would obtain a non--zero section of $H^0\big(F,\cO_F(2h-2D)\big)$. 

Since $(2h-2D)h=0$ this latter space is non zero if and only if $\cO_F(D)\cong\cO_F(2h-D)$. Since $D^2=0$ and $D(2h-D)=8$, it follows that such an isomorphism cannot exist, thus $h^0\big(F,\cO_F(2h-2D)\big)=h^0\big(F,\cO_F(2D-2h)\big)=0$.

We conclude that the map $\cO_F(D)\to\cO_F(D)$ is non--zero, hence it is the identity map. In particular the surjection $\overline{\cE}\to\cO_F(D)$ has a section, thus $\cO_F(D)$ is a direct summand of $\overline{\cE}$, again a contradiction.

We conclude that the two families obtained above are disjoint.
\end{remark}

Theorem \ref{t2gg} has the following corollary as its immediate consequence.

\begin{corollary}
\label{cPfaffian}
Let $F\subseteq\p3$ be a smooth quartic surface.

Then the equation of $F$ is the pfaffian of a $4\times4$ skew--symmetric matrix with quadratic entries.
\end{corollary}

\begin{remark}
There is an interesting interpretation of the corollary above. Indeed let $\Bbb X\subseteq\vert \cO_{\p3}(4)\vert\cong\p{34}$ be the locus of quartics with equation $q_1q_2=0$ for suitable quadratic forms $q_1,q_2$. We denote by $\Bbb S$ the locus in $\vert \cO_{\p3}(4)\vert$ of smooth quartic surfaces.

Let
$$
\sigma_3^0(\Bbb X):=\bigcup_{x_1,x_2,x_3\in\Bbb X}\langle x_1,x_2,x_3\rangle\subseteq \vert \cO_{\p3}(4)\vert
$$
where $\langle x_1,x_2,x_3\rangle$ is the linear span inside $\vert \cO_{\p3}(4)\vert$ of the points $x_1,x_2,x_3$. The $3$--secant variety  $\sigma_3(\Bbb X)$ of $\Bbb X$ is the closure of $\sigma_3^0(\Bbb X)$.

It is well--known that $\sigma_3(\Bbb X)$ fills the whole space (see \cite{C--C--G2}). This means in particular that there is an open subset $\mathcal U\subseteq \vert \cO_{\p3}(4)\vert$ whose elements are in $\sigma_3^0(\Bbb X)$. The above Corollary \ref{cPfaffian} shows that $\mathcal U\supseteq\Bbb S$.
\end{remark}

\section{Construction of bundles with $n=6$}
\label{sConstruction2ngg}
Each twisted cubic $C\subseteq\p3$ obviously intersects $F$ in a scheme $X$ of degree $12$. If $E\subseteq X$ is a scheme of degree $8$, then the corresponding residual scheme $Y$ inside $X$ is trivially the zero--locus of a section of $\omega_C(2h)$, because $\omega_C\cong\cO_{\p1}(-2)$.
Thanks to Lemma 5.4  of \cite{K--M--MR--N--P}, we conclude that $E$ is aG. 

The Artin reduction $R$ of the homogeneous coordinate ring of $E$ has length $8$ and its Hilbert function is symmetric (because it is Gorenstein and graded). Thus the possible Hilbert functions for $R$ are in the List \eqref{List}. 

In the fourth (resp. third) case $E$ would be contained in a line (resp. a plane), because $\dim_k((S_E)_1)=1$ (resp. $2$). This is not possible because a twisted cubic has no trisecant (resp. a subscheme of $C$ contained in a plane has at most degree $3$). The first case cannot occur too, because it would imply $\dim_k((S_E)_1)=7$.

It follows that the second case occurs, hence the cohomology of Sequence \eqref{seqInclusion} gives 
$$
h^0\big(F,\cI_{E\vert F}(h)\big)\le h^0\big(F,\cI_{E}(1)\big)=h^0\big(\p3,\cO_{\p3}(1)\big)-\dim_k((S_E)_1)=0.
$$
In particular there is an initialized aCM bundle $\cE$ of rank $2$ fitting into Sequence \eqref{seqIdeal} due to Proposition \ref{pExistence}. Looking at the Hilbert function of $S_E$ we deduce that $h^0\big(\p3,\cI_E(2)\big)=3$. Thanks to a well--known result (see \cite{B--E}) the homogeneous ideal $I_E\subseteq S$ is generated by the $(\delta-1)\times(\delta-1)$ pfaffians of a skew--symmetric matrix $\Delta$ of odd order $\delta$ and there is a minimal free resolution of the form
\begin{align*}
0\longrightarrow S(-6)\longrightarrow &S(-4)^{\oplus3}\oplus S(-3)^{\oplus\beta}\oplus L^\vee(-6)\longrightarrow \\
&\mapright{\Delta}S(-2)^{\oplus3}\oplus S(-3)^{\oplus\beta}\oplus L\longrightarrow I_E\longrightarrow0
\end{align*}
where $L$ is a direct sum of $S(-j)$'s with $j\ge4$. Since the resolution is minimal, it follows that $L=0$, hence $\beta$ must be even. 
In particular we know that
$$
\Delta=\left(\begin{array}{cc}
B&A\\
-{}^tA&0
\end{array}
\right)
$$
where $B$ is a $3\times3$ skew--symmetric matrix with quadratic entries, $A$ is a $3\times \beta$ matrix with linear entries and $0$ is the $\beta\times\beta$ zero--matrix. If $\beta\ge4$, then the subpfaffians of $\Delta$ obtained by deleting one of the three first rows and columns should vanish, thus the above resolution would not be minimal. We deduce that $\beta\le2$. 

We conclude that either $\beta=0$ and the scheme $E$ is the complete intersection of three quadrics, or $\beta=2$ and $E$ is the degeneracy locus of a skew-symmetric matrix of the form
\begin{equation*}
\left(\begin{array}{ccccc}
0&b_1&b_2&a_1&a_4\\
-b_1&0&b_3&a_2&a_5\\
-b_2&-b_3&0&a_3&a_6\\
-a_1&-a_2&-a_3&0&0\\
-a_4&-a_5&-a_6&0&0
\end{array}
\right),
\end{equation*}
where $a_i$ and $b_j$ are respectively linear and quadratic forms.

Notice that $h^0\big(F,\cI_{E\vert F}(2)\big)=3$, hence the quadrics containing $E$ are the same quadrics generating the ideal of $C$. It follows that we are necessarily in the second case. 

\begin{theorem}
\label{t2ngg}
Let $F\subseteq\p3$ be a smooth quartic surface.

There exist bundles $\cE$ of rank $2$ on $F$ whose minimal free resolution on $\p3$ is Resolution \eqref{resGeneral} with $n=6$. 

The general bundle fitting in Resolution \eqref{resGeneral} with $n=6$ is indecomposable. Decomposable bundles are exactly the direct sums $\cO_F(D)\oplus\cO_F(2h-D)$ where only one of the following occurs:
\begin{enumerate}
\item $D$ and $2h-D$ are aCM quartic curves with $p_a(D)=p_a(2h-D)=1$ which are not intersection of quadrics;
\item $D$ is an irreducible, projectively normal, quintic curve of genus $2$ and $2h-D$ is an aCM, cubic curve with $p_a(2h-D)=0$. 
\end{enumerate}
\end{theorem}
\begin{proof}
We proved above the existence of bundles $\cE$ fitting into Resolution \eqref{resGeneral} with $n=6$ starting from a scheme of degree $8$ which is not the complete intersection of three quadrics.

Assume that $\cE$ is decomposable. We can write $\cE\cong\cO_F(D)\oplus\cO_F(2h-D)$ for some divisor $D$ on $F$, hence both $\cO_F(D)$ and $\cO_F(2h-D)$ are aCM. Since $\cE$ is initialized, it follows that they are either both initialized, or one of them is initialized and the other one has no sections. Moreover $(2h-D)D=c_2(\cE)=8$, hence $hD=4+D^2/2$.

Assume that $\cO_F(D)$ is initialized. We claim that the case $D^2=4$ cannot occur. Indeed, in this case, $p_a(D)=3$ and $D$ is projectively normal (see Lemma \ref{lSextic}). Proposition 6.2 of \cite{Bea} yields the existence of Resolutions \eqref{seqDeterminantal} and \eqref{seqDeterminantalDual}.
Their direct sum, which is Resolution \eqref{resGeneral} with $n=8$, gives a minimal free resolution of $\cO_F(D)\oplus\cO_F(2h-D)$, a contradiction. 

We conclude that $D^2=-2,0,2$, hence $h^0\big(F,\cO_F(D)\big)=2+D^2/2\le 3$. We deduce that $h^0\big(F,\cO_F(2h-D)\big)\ge1$ thanks to Equality \eqref{sumCohomology}, i.e. $\cO_F(2h-D)$ is initialized too, thus it appears in the list of Proposition \ref{pWa}. Moreover $\cE$ is not globally generated: up to permuting the roles of $D$ and $2h-D$, we can assume that the same holds for $\cO_F(2h-D)$. 

Taking into account equalities $hD=4+D^2/2$, $hD+h(2h-D)=8$ and Lemmas \ref{lQuintic}, \ref{lQuartic} we obtain the list of possible cases in the statement.

In particular, if $F$ does not contain curves of the type listed above, then each bundle of rank $2$ fitting into Resolution \eqref{resGeneral} with $n=6$ (whose existence was proved above) is indecomposable.

Conversely let $F$ contain an aCM curve $D$ which is either an irreducible smooth quintic curve of genus $2$, or a cubic curve with $p_a(D)=0$, or a quartic curve with $p_a(D)=1$ which is not intersection of quadrics.

If $D$ is an irreducible smooth quintic curve of genus $2$, then $D^2=2$ (see the second Equality \eqref{h^1}) and $Dh=5$. The Hilbert polynomial of $D$ is $5t-1$, thus $D$ is contained in a unique quadric surface. It follows that $\vert 2h-D\vert\ne\emptyset$ and we trivially have  $(2h-D)h=3$ and $(2h-D)^2=-2$, i.e. $F$ contains also a cubic curve of arithmetic genus $0$. 

In Section \ref{sQuartic} (see, in particular, the proof of Lemma \ref{lQuintic}) we showed that if $F$ contains an aCM, cubic curve $D$ with $p_a(D)=0$, then $\vert 2h-D\vert$ contains also an irreducible, projectively normal, quintic curve of genus $2$.

Finally Lemma \ref{lQuartic} shows that if $F$ contains an aCM, quartic curve with $p_a(D)=1$ which is not intersection of quadrics, then $\vert 2h-D\vert$ contains a curve enjoying the same property.

In all the above cases $\cO_F(h-D)$ is not effective, because none of the above curves is contained in a plane. We also have $(D-h)h\le 1$, thus $\vert D-h\vert$ is either a line or it is empty. Since $(D-h)^2=-4$, the latter case occurs. We conclude that both $\cO_F(D)$ and $\cO_F(2h-D)$ are initialized. 

Again Equality \eqref{RRgeneral}  implies $h^1\big(F,\cO_F(2D-2h)\big)\ge6$ in each case, thus it follows the existence of non--split Sequences 
\begin{equation}
\label{seqSplitNGG}
0\longrightarrow\cO_F(D)\longrightarrow\cE\longrightarrow\cO_F(2h-D)\longrightarrow0,\qquad Dh=4,5,\ p_a(D)=1,2.
\end{equation}

Assume that $\cE$ fits into Sequence \eqref{seqSplitNGG}. If $\cE\cong \cO_F(\overline{D})\oplus\cO_F(2h-\overline{D})$ and $\overline{D}h=Dh$, then imitating word by word the corresponding part of the proof of Theorem \ref{t2gg} one checks that the sequence splits as well, a contradiction. If $\overline{D}h\ne Dh$, then each morphism $\cO_F(D)\to\cE$ is necessarily zero, because $(2h-\overline{D}-D)h=(\overline{D}-D)h=-1$, a contradiction. 

We conclude that extensions as above with an indecomposable sheaf $\cE$ in the middle always exist. Such an $\cE$ is aCM and initialized because the same is true for $\cO_F(D)$ and $\cO_F(2h-D)$: moreover it is immediate to check that $c_1(\cE)=\cO_F(2h)$ and $c_2(\cE)=8$. Finally it is not globally generated due to Theorem \ref{t2gg}. We deduce that it fits in Resolution \eqref{resGeneral} with $n$ either $6$ or $8$.

Assume that the last case occurs. Then $F$ would contain a sextic $A$ and we should have a monomorphism $\cO_F(A)\to\cE$. By composition we obtain a morphism $\cO_F(A)\to\cO_F(2h-D)$. Since $(2h-D-A)h\le-2$, it follows that such a map must be zero, thus $\cO_F(A)\to\cE$ factors through $\cO_F(D)$. Since $(D-A)h\le -2$, we conclude that this morphism must be zero, thus $\cO_F(A)\to\cE$ is zero too, a contradiction. We deduce that $\cE$ fits in Resolution \eqref{resGeneral} with $n=6$.
\end{proof}

\section{$\mu$--Stability and simple bundles}
\label{sStable}
In this section we will deal with the stability properties of an indecomposable, initialized, aCM bundle $\cE$ of rank $2$ on $F$ with $c_1(\cE)=2h$ and $c_2(\cE)=8$. We first recall some facts about the moduli spaces of simple bundles and the various notions of stability and semistability of bundles.

The slope $\mu(\cF)$ and the reduced Hilbert polynomial $p_{\cF}(t)$ of a bundle $\cF$ on $F$ are
$$
\mu(\cF)= c_1(\cF)h/\rk(\cF), \qquad p_{\cF}(t)=\chi(\cF(th))/\rk(\cF).
$$
The bundle $\cF$ is called $\mu$--semistable if for all subsheaves
$\mathcal G$ with $0<\rk(\mathcal G)<\rk(\cF)$
$$
\mu(\mathcal G) \le \mu(\cF),
$$
and $\mu$--stable if the inequality is always strict.
The bundle $\cF$ is called semistable (or, more precisely,
Gieseker--semistable) if for all $\mathcal G$ as above
$$
p_{\mathcal G}(t) \le  p_{\cF}(t),
$$
and (Gieseker) stable again if the inequality is always strict. In order to check the semistability and stability of a bundle one can restrict the attention only to subsheaves such that the quotient is torsion--free.

We have the following chain of implications for $\cF$:
$$
\text{$\cF$ is $\mu$--stable}\ \Rightarrow\ \text{$\cF$ is stable}\ \Rightarrow\ \text{$\cF$ is semistable}\ \Rightarrow\ \text{$\cF$ is $\mu$--semistable.}
$$

Let $\cE$ be a semistable vector bundle of rank $2$ with reduced Hilbert polynomial $p(t)$. The coarse moduli space $\cM_F^{ss}(p)$ parameterizing $S$--equivalence classes of semistable rank $2$ bundles on $F$ with Hilbert polynomial $p(t)$ is non--empty(see Section 1.5 of \cite{H--L} for details about $S$--equivalence of semistable sheaves). We will denote by $\cM_F^{s}(p)$ the open locus inside $\cM_F^{ss}(p)$ of stable bundles.

The scheme $\cM_F^{ss}(p)$ is the disjoint union of open and closed subsets $\cM_F^{ss}(2;c_1,c_2)$ whose points represent $S$--equivalence classes of semistable rank $2$ bundles with fixed Chern classes $c_1$ and $c_2$. Similarly $\cM_F^{s}(p)$ is the disjoint union of open and closed subsets $\cM_F^{s}(2;c_1,c_2)$.

By semicontinuity we can define open loci $\cM_F^{ss,aCM}(2;c_1,c_2)\subseteq \cM_F^{ss}(2;c_1,c_2)$ and $\cM_F^{s,aCM}(2;c_1,c_2)\subseteq \cM_F^{s}(2;c_1,c_2)$ parameterizing respectively $S$--equivalence classes of semistable and stable aCM bundles of rank $2$ on $F$ with Chern classes $c_1$ and $c_2$.

\begin{remark}
We translate below what the various notions of stability introduced above mean for indecomposable, initialized aCM bundle of rank $2$ with $c_1(\cE)=2h$ and $c_2(\cE)=8$ on our smooth quartic surface $F\subseteq\p3$.

Let $\cO_F(D)\subseteq\cE$ be such that $\cE/\cO_F(D)$ is torsion--free. Then there is a zero--dimensional scheme $Z\subseteq F$ and a divisor $D$ on $F$ such that $\cE$ fits into an exact sequence of the form
\begin{equation}
\label{seqUnstable}
0\longrightarrow \cO_F(D)\longrightarrow\cE\longrightarrow\cI_{Z\vert F}(2h-D)\longrightarrow0
\end{equation}
Notice that $\deg(Z)=c_2(\cE(-D))=D^2-2hD+8$. Moreover $\cE(h)$ is globally generated, the same is true for $\cI_{Z\vert F}(3h-D)$. Hence $Z$ is the intersection of the divisors in $\vert 3h-D\vert$ containing it. We can now start our translation. 

Notice that $\mu(\cO_F(D))=Dh$ and $\mu(\cE)=4$. Thus the bundle $\cE$ is $\mu$--semistable (resp. $\mu$--stable) if and only if for each $D$ as above $Dh\le 4$ (resp. $Dh<4$). 

Notice that the reduced Hilbert polynomials of $\cO_F(D)$ and $\cE$ are
$$
p_{\cO_F(D)}(t)=2t^2+Dht+2+\frac{D^2}2, \qquad p_{\cE}(t)=2t^2+4t+2.
$$
It follows that the bundle $\cE$ is semistable (resp. stable) if and only if for each $D$ as above either $Dh<4$, or $Dh=4$ and $D^2\le0$ (resp. $D^2\le-2$).
\end{remark}

We already know that if $\cE$ is $\mu$--stable (resp. semistable), then it is stable (resp. $\mu$--semistable). The lemma below proves that also the converse is true for the bundles we are interested in.

\begin{lemma}
\label{lStable}
Let $F\subseteq\p3$ be a smooth quartic surface and $\cE$ an indecomposable, initialized aCM bundle of rank $2$ on $F$ with $c_1(\cE)=2h$ and $c_2(\cE)=8$. 

The bundle $\cE$ is stable (resp. semistable) if and only if it is $\mu$--stable (resp. $\mu$--semistable).
\end{lemma}
\begin{proof}
It remains to prove that if $\cE$ is stable (resp. $\mu$--semistable), then it is $\mu$--stable (resp. semistable).

The bundle $\cE$ is stable, but not $\mu$--stable if and only if it fits in Sequence \eqref{seqUnstable} with $D$ such that $Dh=4$ and $D^2\le-2$. The bundle $\cE$ is $\mu$--semistable, but not semistable if and only if it fits in Sequence \eqref{seqUnstable} with $D$ such that $Dh=4$ and $D^2\ge2$. Thus we will assume that  $Dh=4$ and $D^2\ne0$ from now on. 

If $Z=\emptyset$, then $8=c_2(\cE)=8-D^2$, hence $D^2=0$, a contradiction.

We have to examine the case $Z\ne\emptyset$. We have $D^2=D^2-2hD+8=\deg(Z)\ge1$, hence $D^2\ge2$. It follows that $\cE$ cannot be stable in this case. 

Let $\cE$ be $\mu$--semistable, but not semistable, so that $D^2\ge2$. If $h^0\big(F,\cI_{Z\vert F}(h-D)\big)\ge1$, then $Z$ would be contained in a divisor in $\vert h-D\vert$. Thus $D^2-2hD+8\le (h-D)(3h-D)=D^2-4$, whence $hD\ge6$, a contradiction. Thus we can assume $h^0\big(F,\cI_{Z\vert F}(h-D)\big)=0$.

The bundle $\cE$ is initialized and aCM, thus the Cohomology of Sequence \eqref{seqUnstable} gives
$$
h^1\big(F,\cO_F(D-h)\big)=h^0\big(F,\cO_F(D-h)\big)=0.
$$
It follows that $D-h$ is not effective. Since $(D-h)h=0$, it follows that $\cO_F(D-h)\ne\cO_F$, hence $h-D$ is not effective too. In particular
$$
h^2\big(F,\cO_F(D-h)\big)=h^0\big(F,\cO_F(h-D)\big)=0.
$$
Equality \eqref{RRgeneral} for $\cO_F(D-h)$ implies $D^2=0$, a contradiction. 
\end{proof}

If $\cE$ fits into Sequences either \eqref{seqSplitExotic}, or  \eqref{seqSplitNGG} with $Dh=5$, then $\cE$ is trivially not $\mu$--semistable. 

Assume that $\cE$ fits into Sequence either \eqref{seqSplitGG}, or \eqref{seqSplitNGG} with $Dh=4$. The bundle $\cE$ is certainly not $\mu$--stable. If it is not $\mu$--semistable, then there is a divisor $A$ on $F$ with $Ah\ge 5$ and an injective morphism $\psi\colon\cO_F(A)\to\cE$. By composition we obtain a morphism $\cO_F(A)\to\cO_F(D)$ which is zero because $(D-A)h\le -1$. Thus $\psi$ factors through a morphism $\cO_F(A)\to\cO_F(2h-D)$ which is again zero because $(2h-D-A)h\le -1$. From the contradiction we deduce that such an $A$ cannot exist. We conclude that $\cE$ is $\mu$--semistable, hence actually semistable due to Lemma \ref{lStable}

The above discussion proves the \lq if\rq\ part of the two assertions of the following statement.

\begin{proposition}
\label{pStable}
Let $F\subseteq\p3$ be a smooth quartic surface and $\cE$ an indecomposable, initialized aCM bundle of rank $2$ on $F$ with $c_1(\cE)=2h$ and $c_2(\cE)=8$. 

The bundle $\cE$ is strictly semistable (resp. unstable) if and only if it fits into Sequence either \eqref{seqSplitGG}, or \eqref{seqSplitNGG} with $Dh=4$ (resp.  either \eqref{seqSplitNGG} with $Dh=5$, or  \eqref{seqSplitExotic}).
\end{proposition}
\begin{proof}
It remains to prove the \lq only if\rq\ part of the statement. In view of Lemma \ref{lStable} it suffices to prove the statement with $\mu$--semistable instead of semistable.
 
If $\cE$ is strictly $\mu$--semistable, then it fits in Sequence \eqref{seqUnstable} with $Dh=4$. By repeating verbatim the proof of Lemma \ref{lStable} one obtains $D^2=0$ and $Z=\emptyset$. Hence the statement is proved in this case.

If $\cE$ is strictly unstable, then it fits in Sequence \eqref{seqUnstable} with $Dh\ge5$. We first assume that $Z\ne\emptyset$. Since $D^2$ is even and $D^2-2hD+8=\deg(Z)\ge1$, it follows that $D^2\ge2hD-6\ge4$. We have  $h^2\big(F,\cO_F(D)\big)=h^0\big(F,\cO_F(-D)\big)=0$, because $-hD\le-5$, thus 
$$
h^0\big(F,\cE\big)-h^1\big(F,\cO_F(D)\big)\ge h^0\big(F,\cO_F(D)\big)-h^1\big(F,\cO_F(D)\big)=2+\frac{D^2}2\ge4.
$$
Taking into account that $h^0\big(F,\cE\big)=4$ we deduce that $h^0\big(F,\cO_F(D)\big)=4=D^2$ and $hD=5$. The bundle $\cE$ is initialized and aCM, thus the Cohomology of Sequence \eqref{seqUnstable} gives
$$
h^1\big(F,\cO_F(D-h)\big)=h^0\big(F,\cI_{Z\vert F}(h-D)\big)\le h^0\big(F,\cO_F(h-D)\big)=0,
$$
because $(h-D)h\le-1$. For the same reason $h^2\big(F,\cO_F(D-h)\big)=h^0\big(F,\cO_F(h-D)\big)=0$, hence Equality \eqref{RRgeneral} gives $1= h^0\big(F,\cO_F(D-h)\big)\le h^0\big(F,\cE(-h)\big)$ contradicting the hypothesis that $\cE$ is initialized.

Finally let $Z=\emptyset$. In this case $8=c_2(\cE)=2hD-D^2$. As above we have $h^2\big(F,\cO_F(D)\big)=0$ and $4=h^0\big(F,\cE\big)\ge h^0\big(F,\cO_F(D)\big)$, thus $4\ge 2+{D^2}/2$ or, equivalently $D^2\le4$. It follows that either $D^2=2$ and $Dh=5$, or $D^2=4$ and $Dh=6$. Since $0=h^0\big(F,\cE(-h)\big)\ge h^0\big(F,\cO_F(D-h)\big)$, Proposition \ref{pWa} implies that $\cO_F(D)$ is aCM, hence Sequence \eqref{seqUnstable} coincides with Sequence either \eqref{seqSplitNGG} with $Dh=5$, or \eqref{seqSplitExotic}.
\end{proof}

\begin{corollary}
\label{cSimple}
Let $F\subseteq\p3$ be a smooth quartic surface. 

If $\cE$ is an indecomposable, initialized aCM bundle of rank $2$ on $F$ with $c_1(\cE)=2h$ and $c_2(\cE)=8$, then $\cE$ is simple.
\end{corollary}
\begin{proof}
If $\cE$ is stable, then it is simple (see \cite{H--L}, Corollary 1.2.8).

Assume that $\cE$ is not stable. Due to Proposition \ref{pStable}, the bundle $\cE$ fits into
\begin{equation}
\label{seqSplit}
0\longrightarrow\cO_F(D)\longrightarrow\cE\longrightarrow\cO_F(2h-D)\longrightarrow0
\end{equation}
where $Dh=4,5,6$ and $p_a(D)=Dh-3$ (which is equivalent to $D^2=2Dh-8$ using Equality \eqref{h^1}).

Notice that $(2h-2D)h<0$ when $Dh=5,6$, hence $h^0\big(F,\cO_F(2h-2D)\big)=0$ in these cases. The same vanishing holds also for $Dh=4$ as pointed out in Remark \ref{rTwiceGG}. 
The cohomology of Sequence \eqref{seqSplit}, twisted by $\cO_F(-D)$, thus implies $h^0\big(F,\cE(-D)\big)=1$. 

Assume that Sequence \eqref{seqSplit} does not split. The boundary map
$$
H^0\big(F,\cO_F\big)\longrightarrow \ext^1_F\big(\cO_F(2h-D),\cO_F(D)\big)\cong H^1\big(F,\cO_F(2D-2h)\big)
$$
is injective. Thus if we now consider the cohomology of Sequence \eqref{seqSplit} twisted by $\cO_F(D-2h)$, it follows that $h^0\big(F,\cE(D-2h)\big)=h^0\big(F,\cO_F(2D-2h)\big)$.

We have $(2D-2h)D\le -4$. Hence each effective divisor in $\vert 2D-2h\vert$ splits as the sum of $D$ and an effective divisor in $\vert D-2h\vert$. The inequality $(D-2h)h\le -2$ implies that such a divisor cannot exist. We conclude that $h^0\big(F,\cO_F(2D-2h)\big)=0$.

Twisting Sequence \eqref{seqSplit} by $\cE^\vee\cong\cE(-2h)$ and taking its cohomology we finally deduce 
$h^0\big(F,\cE\otimes\cE^\vee\big)=1$.
\end{proof}

Notice that indecomposable bundles fitting into Sequences \eqref{seqSplitGG}, \eqref{seqSplitNGG}, \eqref{seqSplitExotic}, form families isomorphic to $\p5$. We proved this fact in the last case of Theorem \ref{t2Exotic}. The same proof holds also for bundles fitting in Sequences \eqref{seqSplitNGG} and \eqref{seqSplitGG} (in this last case we also need the vanishing $h^0\big(F,\cO_F(2h-2D)\big)=h^0\big(F,\cO_F(2D-2h)\big)=0$ proved in Remark \ref{rTwiceGG}).

\begin{theorem}
\label{tStable}
Let $F\subseteq\p3$ be a smooth quartic surface.

The moduli space $\cM_F^{s,aCM}(2;2h,8)$ is non--empty, smooth and of dimension $10$. The points in $\cM_F^{ss,aCM}(2;2h,8)\setminus\cM_F^{s,aCM}(2;2h,8)$ are in one--to--one correspondence with the set of unordered pairs $\{\ \cO_F(D),\cO_F(2h-D)\ \}$ where $D$ is a quartic curve  with arithmetic genus $1$.
\end{theorem}
\begin{proof}
It is possible to define the moduli space $\Spl_F(2;2h,8)$ of rank $2$ simple bundles $\cE$ on $F$ with $c_1(\cE)=2h$ and $c_2(\cE)=8$. If $\Spl_F(2;2h,8)$ is non--empty, then it is smooth and it has dimension $10$ (e.g see  \cite{Mu}).

By semicontinuity we can define the open locus $\Spl_F^{aCM}(2;2h,8)\subseteq \Spl_F(2;2h,8)$ of aCM bundles. Corollary \ref{cSimple} implies that each rank $2$, aCM bundle $\cE$ on $F$ with $c_1(\cE)=2h$ and $c_2(\cE)=8$ represents a point inside $\Spl_F^{aCM}(2;2h,8)$.

Assume that $\cE$ represents a point in $\Spl_F^{aCM}(2;2h,8)$ which is not stable. We know that it fits into Sequence \eqref{seqSplit} for some divisor $D$ with $Dh=4,5,6$ and $D^2=2Dh-8$. Moreover $\cE$ varies in a flat family $\frak E_D$ over $\p5$ due to the discussion above. In particular we have a morphism $\p5\to \Spl_F^{aCM}(2;2h,8)$ depending on $D$ and we denote by $X_D$ its image. 

We have an inclusion
$$
\bigcup_{D}X_D\subseteq  \Spl_F^{aCM}(2;2h,8).
$$
Recall that the natural quotient map $\Pic(F)\to\Num(F)$ factors through $\NS(F)$, thus it induces an isomorphism $\Pic(F)\cong\NS(F)$ (see \cite{SD}, Paragraph (2.3)). It follows that $\Pic(F)$ is a finitely generated free abelian group, thus it is countable: in particular the set of isomorphism classes of line bundles $\cO_F(D)$, where $D$ is as above, is countable as well.

We deduce that the first member of the above inclusion is a countable union of proper subvarieties of $\Spl_F^{aCM}(2;2h,8)$, hence the inclusion above is strict. It follows that there are stable bundles inside $\Spl_F^{aCM}(2;2h,8)$. Since this moduli space contains $\cM_F^{s,aCM}(2;2h,8)$ as an open subset, we deduce that this last space is non--empty and smooth, of dimension $10$.

Taking into account Proposition \ref{pStable} we know that $\cE$ represents an $S$--equivalence class in $\cM_F^{ss,aCM}(2;2h,8)\setminus\cM_F^{s,aCM}(2;2h,8)$ if and only if it fits into Sequence \eqref{seqSplit} with $Dh=4$.

Notice that $0\subseteq\cO_F(D)\subseteq \cE$ is the Jordan--H\"older filtration for $\cE$. Indeed
$$
\cE/\cO_F(D)\cong\cO_F(2h-D),\qquad p_{\cE}(t)=p_{\cO_F(D)}(t)=p_{\cO_F(2h-D)}(t).
$$
We conclude that  $\gr(\cE)\cong\cO_F(D)\oplus\cO_F(2h-D)$. Hence all the bundles fitting in the aforementioned Sequence \eqref{seqSplit} are $S$--equivalent and they correspond to the same point in $\cM_F^{ss,aCM}(2;2h,8)$. If $\overline{D}$ is another divisor on $F$ with $\overline{D}h=4$ and $\overline{D}^2=0$, it is easy to check that $\cO_F(D)\oplus\cO_F(2h-D)\cong\cO_F(\overline{D})\oplus\cO_F(2h-\overline{D})$ if and only if $\overline{D}$ is linearly equivalent to either $D$, or $2h-D$.
\end{proof}

\bigskip
\noindent
Gianfranco Casnati,\\
Dipartimento di Scienze Matematiche, Politecnico di Torino,\\
c.so Duca degli Abruzzi 24, 10129 Torino, Italy\\
e-mail: {\tt gianfranco.casnati@polito.it}

\bigskip
\noindent
Roberto Notari, \\
Dipartimento di Matematica \lq\lq Francesco Brioschi\rq\rq, Politecnico di Milano,\\
via Bonardi 9, 20133 Milano, Italy\\
e-mail: {\tt roberto.notari@polimi.it}


\begin{thebibliography}{44}

\bibitem{A--C--G--H}
E. Arbarello, M. Cornalba, P.A. Griffiths, J. Harris:
\emph{Geometry of algebraic curves}. Vol I,
Springer,
1985.

\bibitem{Bea}
  A. Beauville: \emph{Determinantal hypersurfaces}.  Michigan Math. J. \textbf{48} \rm(2000), 39--64.
  
\bibitem{B--E}
D.A. Buchsbaum, D. Eisenbud: \emph{Algebra structures for finite free resolutions, and some structure theorems for ideals of codimension
$3$}. Amer. J. Math. \textbf{99} \rm (1977), 447--485.

\bibitem{C--C--G2}
E. Carlini, L. Chiantini, A.V. Geramita: \emph{Complete intersection points on general surfaces in $\p3$}. Michigan Math. J.,
\textbf{59} (2010), 269--281.

\bibitem{C--H1}
  L. Casanellas, R. Hartshorne: \emph{ACM bundles on cubic surfaces}. J. Eur. Math. Soc. \textbf{13} \rm (2011), 709--731.

\bibitem{C--H2}
  M. Casanellas, R. Hartshorne, F. Geiss, F.O. Schreyer: \emph{Stable Ulrich bundles}. Int. J. of Math. \text{23} 1250083 \rm(2012).

\bibitem{Cs1}
G. Casnati: \emph{On rank two aCM bundles}. Preprint.

\bibitem{Cs2}
G. Casnati: \emph{Rank two aCM bundles on general determinantal quartic surfaces in $\p3$}. Preprint 	arXiv:1601.02911 [math.AG]. To appear in Ann. Univ. Ferrara Sez. VII Sci. Mat..

\bibitem{Ct}
F. Catanese: \emph{BabbageÕs conjecture, contact of surfaces, symmetric determinantal varieties and applications}. Invent. Math. \textbf{63} \rm (1981), 433--465.

\bibitem{C--K--M1}
  E. Coskun, R.S. Kulkarni, Y. Mustopa: \emph{Pfaffian quartic surfaces and representations of Clifford algebras}. Doc. Math.  \textbf{17} \rm (2012), 1003--1028.

\bibitem{C--K--M2}
  E. Coskun, R.S. Kulkarni, Y. Mustopa: \emph{The geometry of Ulrich bundles on del Pezzo surfaces}. J. Algebra \textbf{375} (2013), 280--301.

\bibitem{E--He}
D. Eisenbud, D. Herzog: \emph{The classification of homogeneous Cohen--Macaulay rings of finite representation type}. Math. Ann. \textbf{280} \rm(1988), 347--352.

\bibitem{Fa}
D. Faenzi: \emph{Rank $2$ arithmetically Cohen--Macaulay bundles on a nonsingular cubic surface}. J. Algebra \textbf{319} (2008), 143--186.

\bibitem{Ha2}
  R. Hartshorne: \emph{Algebraic geometry}. G.T.M. 52, Springer \rm (1977).

\bibitem{Hu}
D. Huybrechts: \emph {Lectures on K3 surfaces}. To appear: the draft notes was freely available at {\tt http://www.math.uni-bonn.de/people/huybrech/K3Global.pdf}\ .

\bibitem{H--L}
D. Huybrechts, M. Lehn: \emph {The geometry of moduli spaces of sheaves. Second edition}. Cambridge Mathematical Library, Cambridge U.P. \rm (2010).

\bibitem{K--M--MR--N--P}
J.O. Kleppe, J. Migliore, R. Mir\'o-Roig, U. Nagel, Ch. Peterson: \emph{Gorenstein liaison, complete intersection liaison invariants and unobstructedness}. Mem. Amer. Math. Soc. \textbf{154} \rm (2001).

\bibitem{Kr}
M. Kreuzer: \emph{On $0$--dimensional complete intersections}.  Math. Ann. \textbf{292} \rm  (1992), 43--58.

\bibitem{Ma1}
  C. Madonna: \emph{A splitting criterion for rank $2$ vector bundles on hypersurfaces in $\bP^4$}. Rend. Sem. Mat. Univ. Pol. Torino \textbf{ 56} \rm (1998), 43--54.

\bibitem{Mi}
J.C. Migliore: \emph{Introduction to liaison theory and deficiency modules}. Progress in Mathematics 165, Birkh\"auser (1998).

\bibitem{Mu}
S. Mukai: \emph{Symplectic structure of the moduli space of sheaves on an abelian or $K3$ surface
}.  Invent. Math. \textbf{77} (1984), 101--116.

\bibitem{O--S--S}
  C. Okonek, M. Schneider, H. Spindler: \emph{ Vector bundles on complex projective spaces}. Progress in Mathematics 3, \rm(1980).

\bibitem{Ot}
  G. Ottaviani: \emph{Spinor bundles on quadrics}. Trans. A.M.S. \textbf{ 307} \rm(1988), 301--316.

\bibitem{SD}
B. Saint--Donat: \emph{Projective models of $K$--$3$ surfaces}. Amer. J. Math. \textbf{96} \rm(1974), 602--639.

\bibitem{Wa}
K. Watanabe: \emph{The classification of ACM line bundles on a quartic hypersurface on $\p3$}. Geom. Dedicata
\textbf{175} \rm (2015) 347--353.



\end{thebibliography}
\end{document}